\documentclass[letterpaper, 10 pt, conference]{ieeeconf}  

\IEEEoverridecommandlockouts                              
\usepackage[utf8]{inputenc}
\usepackage[T1]{fontenc}

\usepackage[final]{hyperref}
\hypersetup{
    colorlinks=true,
    linkcolor=black,
    filecolor=magenta, urlcolor=blue,
    citecolor=blue,
    } %
\usepackage{indentfirst}
\usepackage{xcolor}
\usepackage{graphicx}
\usepackage{amsmath}
\usepackage{float}
 \usepackage{amssymb}

\newcommand{\sign}{\operatorname{sign}}
\newcommand{\sgn}{\operatorname{sgn}}
\newtheorem{proposition}{Proposition}
\newtheorem{remark}{Remark} 
\newtheorem{theorem}{Theorem}

\title{\LARGE \bf
Multivariable Extremum Seeking Control for Dynamic Maps through Sliding Modes and Periodic Switching Function
}

\author{Nerito Oliveira Aminde$^{1}$,  Tiago Roux Oliveira$^{2}$ and Liu Hsu$^{3}$
\thanks{This study was financed in part by the Coordenação de Aperfeiçoamento de
Pessoal de Nível Superior - Brasil (CAPES) Finance Code 001; Conselho Nacional de Desenvolvimento Científico e Tecnológico, CNPq; Fundação de Amparo à Pesquisa do Estado do Rio de Janeiro, FAPERJ;}
\thanks{$^{1}$N. O. Aminde is with Federal Center for Technological
Education "Celso Suckow da Fonseca"
(CEFET/RJ), Angra dos Reis, Rio de Janeiro
– RJ, 23953-030, Brazil
        {\tt\small nerito.aminde@cefet-rj.br}}%
\thanks{$^{2}$T. R. Oliveira is with Department of Electronics and Telecommunication Engineering, State University of Rio de Janeiro (UERJ), Brazil
        {\tt\small tiagoroux@uerj.br}}%
 \thanks{$^{3}$L. Hsu is with Department of Electrical Engineering, Federal University of Rio de Janeiro (COPPE/UFRJ), Rio de Janeiro, Brazil
        {\tt\small lhsu@coppe.ufrj.br}}%
}

\begin{document}

\maketitle
\thispagestyle{empty}
\pagestyle{empty}


\begin{abstract}

  This paper presents the design of an extremum seeking controller based on sliding modes and cyclic search for real-time optimization of non-linear multivariable dynamic systems. These systems have arbitrary relative degree, compensated by the technique of time-scaling. The resulting approach guarantees global convergence of the system output to a small neighborhood of the optimum point. To corroborate with the theoretical results, numerical simulations are presented considering a system with two inputs and one output, which rapidly converges to the optimal parameters of the objective function.

\vspace{0.2cm}
\end{abstract}

\section{INTRODUCTION}

Extremum seeking control is a form of adaptive control whose objective is to find input from a control system that maximizes or minimizes its output in real-time  \cite{AK:03}. Since its emergence in 1922, it has already been applied to ABS brakes, mobile robots, vehicles and even particle accelerators \cite{NMMM:10}. Recently, the scope of its applications has been expanded, such as ESC analysis in the presence of delays \cite{OK:22}, implemented in neuromuscular stimulation problems \cite{POPF2020}, non-cooperative games \cite{ORKBEK:19,Basar_heat,Basar_heterogeneous}, biological reactors \cite{RPFT:19,OFKK2020} and traffic control for urban mobility \cite{YSOK:2020}. 

Most of the publications in extemum seeking control concentrates on single-input-single-ouput (SISO) systems \cite{Emilia_TAC,Damir_LCSS,Newton_TDS,Damir_EJC,PDE_cascades_SCL,RAD_PDE_ALCOS}. However, many of the real-life problems and applications that require optimization involve multivariable systems \cite{AK:03}. Therefore, in recent years, several techniques of multivariable control systems were proposed \cite{GKN:12,XLS:14,TM:17}.

In general, there are several methodologies available to perform extremum seeking control. One of the most recognized approaches is based on periodic excitation signals or dithers  \cite{KW:00}. Another strategy involves extremum seeking control through monitoring functions and sliding modes \cite{AOH:13,AOH:14}, where the monitoring function is used to address the lack of knowledge of the control direction.

Recently, a multivariable extremum seeking control based on periodic switching function and sliding modes was proposed in \cite{SO:18}, in which the concept of cyclic search was introduced. Basically, this approach reduces a multivariable problem in scalar sub-problems. In \cite{AOH:21}, we expand this formulation to consider the monitoring function for static objective  function \cite{AOH:20} and dynamic linear systems \cite{AOH:14}, respectively.

In this paper, we propose a multivariable extremum seeking controller based on periodic switching function \cite{TAC_Guo} and sliding modes for dynamic maps with arbitrary relative degrees, which for simplicity are considered stable linear dynamic systems in cascade with a nonlinear map. The relative degree is mitigated through time-scaling \cite{MED_2017}, which also allows the analysis and design of controllers by extremum seeking regardless of the order or relative degree of the model and exact knowledge of its parameters \cite{AOH:14}.

\section{Problem Statement}
\label{formulaprob}

Consider the following uncertain linear subsystem with arbitrary relative degree $n^*$:
\begin{align}
    \dot{v}&=u,\label{sistinteg} \\
    \dot{x}&=Ax+Bv,\label{dinamicalinear} \\
          z&=Cx \label{subsist1}
\end{align}
in cascade to the static system
\begin{align}
    y=h(z),\label{saidamensura}
\end{align}
where $u\in \mathbb{R}^m$ is the control input, $x\in \mathbb{R}^n$ is the state vector, $z\in \mathbb{R}^n$ is the unmeasured output of the subsystem (\ref{sistinteg})-(\ref{subsist1}) and $y\in \mathbb{R}$ is the measured output (\ref{saidamensura}). 
 
The integrator in (\ref{sistinteg}) is used to obtain a virtual control signal $v \in \mathbb{R}^n$, which increases the relative degree of the system \cite{L:03}, i.e., $n \geq n^*-1$ instead of $n>n^*$. The increase in the relative degree causes the high-frequency switching to be retained only in the control signal $u$, while the  virtual control $v$ that triggers the plant is totally continuous, which allows attenuating the chattering \cite{Uetal:99} in the closed-loop system.

The matrices $A \in \mathbb{R}^{n \times n}$, $B n \mathbb{R}^{n \times m}$, $C \in \mathbb{R}^{n \times n}$ are uncertain, the subsystem of order $n$ and consequent relative degree is assumed unknown. In order to ensure the existence and uniqueness of solutions, it is assumed that the uncertain nonlinear map $h: \mathbb{R}^n \to \mathbb{R}$ to be optimized (maximized or minimized) is locally Lipschitz continuous and sufficiently smooth. It is also assumed that the initial instant is $t=0$~s. For each solution of (\ref{sistinteg})-(\ref{saidamensura}), there is a maximum time interval of definition given by $[0,t_M)$, where $t_M$ can be finite or infinite.

\subsection{Control objective}

The control objective is to design a control law $u$ through output feedback that takes the system to the extremum point of the objective function $y=h(z)$ in (\ref{saidamensura}), without losing generality and starting from any initial conditions, keeping it as close as possible to this point. The objective function is assumed to have only a maximum, denoted by $y^*=h(z^*)$. This challenge can be conceived under the extremum seeking of a control system, where $y$ represents the output of the objective function, and $v$ is interpreted as the output of an integrator, whose input is determined by a $u$ vector of control signals to be designed.

Furthermore, the same problem can be redefined as a trajectory tracking problem, in which the control direction is unknown \cite{POL:12}.

\subsection{Analysis by Singular Perturbation}
\label{perturbsingular}

In \cite{AOH:22}, the multivariable extremum seeking control was designed via sliding modes, periodic switching function and cyclic search for optimization problems in real-time, but for static maps. Here it is intended to show that the results obtained can be extended to dynamic maps.

To ensure this generalization, initially consider a simple integrator system with a nonlinear static map

\begin{align}
     \dot v &= u\,,                 \label{plant_inverse3}\\ 
       y&=h(v)\,,              \label{saidameasured3}
\end{align}
which can be effectively controlled using extremum seeking control via periodic switching function.

By considering  singular perturbation method
 \cite{KKO:1999}, it can be shown that extremum seeking control via periodic switching and cyclic search \cite{AOH:22} is robust to fast unmodeled dynamics such that the disturbed system
(\ref{plant_inverse3})-(\ref{saidameasured3}) can be rewritten in sensor block form \cite[p.
50]{KKO:1999}:
\begin{eqnarray}             
    \dot v &=& u \,,                               \label{fast_eq0}\\
    \eta\dot{x}&=& Ax+Bv \,,     \label{fast_eq1}  \\
    y &=& h(Cx) \,,  \label{fast_outupunmeasured}
\end{eqnarray}
satisfying the inequality
\begin{eqnarray}             
    |y-y^*| \leq \mathcal{O}(\sqrt{\eta}+\varepsilon)\,,    \label{residual_set}
\end{eqnarray}
where
$\eta>0$ is a constant sufficiently small. The demonstration of (\ref{residual_set}) follows the same steps as in \cite{CH:1991,CH:1992,AOH:22}, considering $y^*$ as a setpoint.

\subsection{Redesign of the controller via time-scaling }

Using a suitable time-scale \cite{AOH:14} 
\begin{equation}
  \frac{dt}{d\tau} = \eta\,,                                              \label{time-scaling}\\
\end{equation}
one can rewrite the system (\ref{fast_eq0})--(\ref{fast_outupunmeasured}) as
\begin{eqnarray}
  v' &=& \eta u                               \label{newsis}\\
  x' &=& A x + Bv\,,                                              \label{plant_inverse2new}\\
  z &=& C x \,,                     \label{plant_extern2new}\\
  y&=& h(z)\,,                                \label{saidameasurednew}
  \end{eqnarray}
where $v':=\frac{d v}{d \tau}$ e $x':=\frac{d x}{d \tau}$.

This means that $\exists \eta^*>0$ such that de control input $u$ can be scaled in (\ref{newsis}) to control the original system
(\ref{dinamicalinear})--(\ref{saidamensura})
in a different time-scale dilated by $t=\eta\tau$,
$\forall \eta \in (0,\eta^*]$.

The practical meaning is that, since the periodic switching function-based extremum seeking control, initially proposed for systems of relative degree one, is robust to unmodeled dynamics that are stable and fast when ($ \eta \to 0$) then it can be effective to control dynamics systems of any relative degree as long as they are properly staggered. As a predictable consequence, the cost to consider is that the closed loop response slows down as $\eta \to 0$.

As mentioned earlier, only the output is considered measurable, while the linear output of the subsystem $z$ and the state $x$ are not available for feedback. To achieve the expected results, it is necessary to assume and delimit some hypotheses presented below, considering the controlled plant (\ref{sistinteg})-(\ref{saidamensura}), or equivalently (\ref{newsis})-(\ref{saidameasurednew}):

\label{hypotesis}
$\bf{(H1)}$ (\emph{About the uncertainties}): All uncertain plant parameters belong to a compact set $\Omega$.

\label{hypotesis2}
$\bf{(H2)}$ (\emph{Differentiability of $h$}): The nonlinear function $h:\mathbb{R}^n \rightarrow \mathbb{R}$ is  locally Lipschitz continuous in $x$ and sufficiently smooth, i.e., $h$ is continuously  differentiable over all domain $\mathbb{R}^n$.

\label{hypotesis3}
$\bf{(H3)}$ (\emph{About the linear subsystem}): the matrix $A$ in (\ref{dinamicalinear}) must have its eigenvalues in the left-half plane, i.e., it must be stable.

\label{hypotesis4}
${\textbf{(H4)}}$ (\emph{Unique maximum in $h(z)$}):  
It is assumed that there is $z^*\in\mathbb{R}^n$ such that $y^*=h(z^*)$ is the only maximum of $h(z)$, where the gradient and hessian matrices satisfy, respectively:
\begin{align*}
\frac{\partial h}{\partial z}\bigg\rvert_{z=z^*}=0\;\; \text{and}\;\;\frac{\partial^2 h}{\partial z^2}\bigg\rvert_{z=z^*}<0, \;  \end{align*}
where $z\in\mathbb{R}^n$.

\label{hypotesis5}
${\textbf{(H5)}}$ (\emph{Radial unboundedness of $h$}): Assume that the function $h:\mathbb{R}^n \rightarrow \mathbb{R}$ is radially unbounded in $\mathbb{R}^n$. This guaranties that if $|y|$ is bounded, then $\|x\|$ must be bounded.

Hypothesis H1 is fundamental and crucial in nonlinear systems. Hypotheses H2, H3 and H5 are required to ensure no finite-time escape in the closed-loop system, while hypothesis H4 demonstrates maximum properties of a nonlinear function, being vital in optimization problems.

\section{Multivariable extremum seeking controller via periodic switching function} \label{section3}

Figure \ref{diagram_dynamic} represents the diagram of the proposed multivariable extremum seeking control scheme, which uses sliding modes and output feedback with a periodic switching function and cyclic search.

\begin{figure}[!htb]
\begin{center}
\includegraphics[width=.47\textwidth]{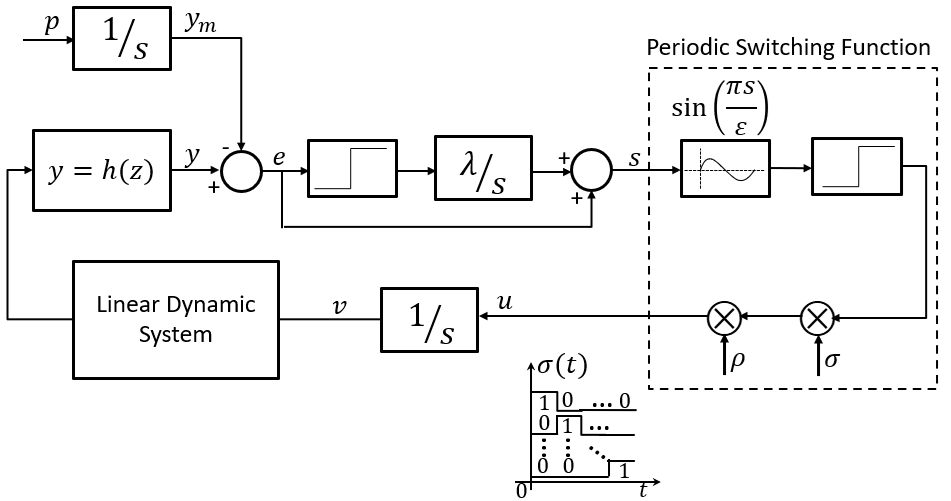}
\caption{Multivariable extremum seeking control via periodic switching function and cyclic search for dynamic maps.}
\label{diagram_dynamic}
\end{center}
\end{figure}
The control law is given by 
 \begin{equation}
u(t)\!=\!\rho(t)\sigma(t)\!\sgn\left(\sin\left[\frac{\pi}{\varepsilon}
s(t) \right] \right)\,,         \label{eqdrakunov}
\end{equation}
such that
\begin{equation}
s(t) = e(t) + \lambda \int_0^t \sgn(e(\tau))d\tau                            \label{smcerror}
\end{equation}
where $\sigma(t)$ is the cyclic search function and $\rho(t)$ is the modulation function, both functions to be defined in sections \ref{cyclicfunction} and \ref{funcaodemodulacao}, respectively, and $\lambda, \varepsilon >0$ are appropriate constants.

 The reference tracking error $e(t)$ is defined by
\begin{equation}
e(t) = y(t) - y_m(t)\,,                 \label{erro}
\end{equation}
where $y_m>0$ is a monotonically increasing ramp as a function of time, generated through the following reference trajectory
\begin{equation}
\dot{y}_m = p,  \quad y_m(0)=p_0 \,,                                               \label{modref}   
\end{equation}
where $p>0$ is design constant. To avoid an unlimited reference signal, $y_m(t)$ can be saturated in the controller by an upper bound of $y^*$, without affecting its performance.

The modulation function $\rho(t)$ will be designed so that $y(t)$ can track the ramp $y_m(t)$, $\forall t$, until the directional or global extremum point is reached. Thus, $y$ is oriented to reach the neighborhood of the directional or global maximum $y^*=h(x^*)$ and to remain close to the ideal value. For this, a new modulation function $\rho(t)$ is suggested, such that the sliding mode $\dot{s}=0$ occurs in finite time on one of the surfaces $s=k\varepsilon$, for some $k$ integer.

From (\ref{smcerror}), one has
\begin{align}
\dot{s}=\dot{e}+\lambda \sgn(e)=0.
\end{align}

Thus, it is ensured that the error $e$ tends to zero, that is, $y=h(z)$ tracks $y_m$ (and consequently, $y$ must approach the directional or global extremum $y^*$) while $y$ remains outside the neighborhood of $y^*$, where the high-frequency gain is non-zero. In contrast, once $y$ approaches $y^*$, the high-frequency gain approaches zero, which results in loss of controllability. This way, tracking of $y_m$ is stopped. However, the neighborhood of the extremum point is reached as desired. The control strategy ensures that $y$ remains close to $y^*$, $\forall t$. Results and demonstrations of the scalar case can be found in \cite{POL:12}.

\subsection{Design of the cyclic directional search}
\label{cyclicfunction}

 The cyclic directional search is designed as in \cite{AOH:20,SO:18}, so that the change of the search direction takes place  periodically. Let $\sigma(t)$ be a periodic function with period $T_s$. Let the  interval $T=[T_l, T_u)$  represent the interval of one cycle so that, $T_s= T_u - T_l$ and let the set of initial instants of each directional search be denoted as  $\tau_1, \tau_2,\cdots,\tau_n$. Let $\Delta \tau_i=[\tau_i, \tau_{i+1}), \forall i=1,...,n$, be the $n$ directional search sub-intervals within each cycle.
Let the orthonormal basis be denoted by $a_1,a_2,\cdots,a_n \in \mathbb{R}^n$, so that $a^T_i=[0,\cdots 0, 1,0, \cdots 0]$, with the  unit element at the $i$th position of the vector. Thus, the cyclic search direction in (\ref{eqdrakunov}) can be defined as follows:
\begin{align}
      \sigma(t)=a_i, \quad \forall t \in \Delta \tau_i,\quad \forall i=1,...,n. 
      \label{dirbusca}
\end{align}

For simplicity, we choose equal duration $T_s/n$ for each sub-interval $\Delta\tau_i$. During the $i$th  sub-interval, the controller will search in the $i$th direction before switching to the next direction in the next sub-interval (\cite{SO:18}). That is, the multivariable cyclic controller works as a scalar controller in each sub-interval $\Delta \tau_i$. 

Since the search is cyclic with period $T_s$, to each cycle is attributed an index $\kappa=1,2,\ldots,\infty$. For the $\kappa$th cycle, a directional extremum can occur for each $i$th search direction within the cycle. Such extremum is denoted $y^*(\kappa)=h(x_i^*(\kappa))$, where $x_i^*(\kappa)\in \mathbb{R}^n$ is an extremum point along the $i$th search direction of the $\kappa$th cycle. When the system approaches the global extremum or some directional extremum, the controllability is ``lost'' in the sense that the control gains become too low to ensure the desired output tracking by sliding mode. Thus, the following concepts and hypothesis are introduced.

\label{hypotesis6}
${\textbf{(H6)}}$ (\emph{Low controllability regions $\mathcal{D}_{\Delta}$ and $\mathcal{D}(\kappa)_{\Delta_i}$}):  

Let $\mathcal{D}_{\Delta}:=\{x:\|x-x^*\|<\Delta/2\}$ and, for each $i$-th directional search $\kappa$-th cycle, indexed by $\kappa_i$,  $\mathcal{D}(\kappa_i)_{\Delta_i}:=
 \{x:\|x_i(\kappa_i)-x^*_i(\kappa_i)\|<\Delta_i/2,\quad  x_j(\kappa_i)=\text{constant},\ j\neq i\}$. Such regions will be referred to as $\Delta$-vicinity, for simplicity.
 
 Then, assume that there exists a 
 positive function $L_h(\cdot)$, bounded away from zero, such that, for any given s.s. constants $\Delta>0$ and $\Delta_i>0$, 
\begin{align}
0<\underline{k}_{p}\leq L_h({\Delta})\leq &\left\|\frac{\partial h}{\partial x}\right\|, \;
\forall x \notin \mathcal{D}_{\Delta},\; \text{and} \notag\\ 0<\underline{k}_{pi} \leq L_h({\Delta_{iJ}})\leq& \left|\frac{\partial h}{\partial x_i}\right|,\;\forall x_i \notin \mathcal{D}(\kappa)_{\Delta_{iJ}}. \;  \notag \end{align}
From the continuity, assumption (\textbf{H1}), $L_h({\Delta})$ and $L_h({\Delta_{iJ}})$ tend to zero as $\Delta$ and $\Delta_{iJ}$ tend to zero. 
Note that the $\Delta$'s can be chosen arbitrarily small if $L_h$ is allowed to be correspondingly small, due to Assumption (\textbf{H1}). However, smaller $L_h$ will demand higher control gain. For simplicity, one can choose $\Delta=\Delta_i$. 

It is convenient to relate the parameter $\Delta$ or $\Delta_i$ with the small parameter $\varepsilon$ so that $||x-x^*||<\sqrt{\varepsilon}$ in $\mathcal{D}_{ \Delta}$ and $ |x-x_i^*(\kappa_i)|\leq \sqrt{\varepsilon}$ in the $i$-th directional domain $\mathcal{D}(\kappa_i)_{\Delta_i}$ , as a consequence of the equation (\ref{fast_outupunmeasured}) \cite{AOH:22}, for static map, $\eta=0$.

In what follows, we drop the periodic search cycle index $\kappa$ to avoid clutter, by considering the analysis takes place within a generic cycle.

\subsection{The singular case  \texorpdfstring{$\eta=0$}{Lg}}

In this case, the differential equation (\ref{fast_eq1}) is replaced by the algebraic equation $x=-A^{-1}Bv$ and, from (\ref{fast_eq0}) and (\ref{fast_outupunmeasured}),
the first derivative of the output $y$ with respect to time is given by
\begin{align}
\dot{y} ={\frac{\partial h}{\partial z}}^Tu
\end{align}
where the high frequency gain is given by the gradient vector, i.e.,
\begin{align}
k_p(z):=\left[k_{p_1} \;\; \cdots \;\; k_{p_n}\right],\;\;\text{e}\quad k_{p_i}(z):={\frac{\partial h}{\partial z_i}CA^{-1}B}. \label{kpx}
\end{align}

As in \cite{POL:12}, the signs associated with the elements $k_{pi}:= \frac{\partial h}{\partial z_i}$ of $k_p$ can be interpreted as control directions. Hypothesis H6 allows the consideration of a non-linear control system with state-dependent high-frequency gain, which changes signal around the optimum point continuously.

From (\ref{kpx}) and hypothesis H6, $k_p$ $(\forall z \notin \mathcal{D}_{\Delta}$), and $k_{p_i}$ $(\forall z \notin \mathcal{D}_{\Delta_{\kappa_i}}$) satisfy 
\begin{align}
0<\underline{k}_{p}\leq \left \|k_{p} \right \|, \ |k_{p_i}| \, ;                                \label{kpbarneras2}
\end{align}
where the lower bound $\underline{k}_{p} \leq  L_h|CA^{-1}B|$ is a constant, considering all permissible uncertainties in $h(.)$, $A$, $B$ e $C$.

\subsection{Design of the modulation function for \texorpdfstring{$\eta=0$}{Lg}}
\label{funcaodemodulacao}

From (\ref{newsis}), (\ref{erro}) and (\ref{modref}), the time derivative of the sliding manifold $s(t)$ (hiding $t$) one has:
\begin{align}
\dot{s}=\sum_{i=1}^n \frac{\partial h}{\partial x_i}u_i-p+\lambda \sign(e)\,,
\end{align}
and for the $i$-th search direction,
\begin{align}
\dot{s}= k_{p_i}(x)(u_i+d_s)\,,  \label{edoerro} 
\end{align}
where
\begin{align}
d_s:=(k_{p_i}(x))^{-1}\left(-p + {\lambda \sgn(e)}\right) \,.         \label{moduerro}
\end{align}

Suppose we start at $t=\tau_i$, that is, at the beginning of the $i$-th search direction, we have controllability error, with $k_{p_i}\geq L_h(\Delta_i)$, to be considered as \textit{controllability condition}
  (see (\textbf{H6})). Considering $d_s$ as a disturbance, it can be increased in absolute value by:
\begin{align}
\bar{d}_s:=L^{-1}_h \left(p + {\lambda}\right) \geq |d_s| \,.         \label{debar}
\end{align}
which allows us to define the modulation function $\rho$ as being
\begin{align}
\rho = \bar{d}_s+\gamma,\;\;
\label{funcmodgen} 
\end{align}
with $\gamma>0$ being an arbitrarily small positive constant. This approach relies on the results obtained and published in \cite{AOH:22}.

\subsection{Modulation function design for \texorpdfstring{$\eta\ne 0$}{Lg}}

For $\eta \ne 0$ in (\ref{fast_eq1}), the time scaling in
(\ref{time-scaling}) allows the original plant to be considered
(\ref{sistinteg})–(\ref{saidamensura}), on a different time-scale,
controlled by the controller (\ref{eqdrakunov}) properly
scaled by $\eta u$ in (\ref{newsis}). In the sense
to incorporate it, the modulation function must be
redesigned to satisfy
\begin{align}
    \rho \geq \eta[|d_s|+\gamma] \label{funcmodgen2}
\end{align}
instead of (\ref{funcmodgen}).

Analyzing the singular perturbation method presented in Section \ref{perturbsingular}, if (\ref{eqdrakunov}) were used again, a limiting
upper bound for the tracking error $e(t)$ could be obtained directly, for $\eta$ sufficiently
small.

\begin{proposition}
 \label{prop:01}
Consider the systems (\ref{sistinteg})--(\ref{saidamensura}),  search direction (\ref{dirbusca}), reference trajectory (\ref{modref}) and control law (\ref{eqdrakunov}). Outside the regions $\mathcal{D}_{\Delta}$ and $\mathcal{D}_{\Delta_i}$, if the modulation function $\rho$ in (\ref{eqdrakunov}) is designed as
\begin{align}
    \rho:=\frac{\eta}{L_h}[p+\lambda]+\eta \gamma\,, \label{funcmodgen3}
\end{align}
satisfying (\ref{funcmodgen2}), while $z \notin \mathcal{D}_{\Delta_i}$, one has:  \textbf{(a)} the sliding mode $s=k\varepsilon$ is reached in regardless of the control and search direction, and \textbf{(b)} there is no escape in finite time ($t_{M} \to +\infty$).

\begin{proof}
Considering the singular perturbation argument and time scaling (\ref{time-scaling}), which shows that the systems (\ref{fast_eq0})-(\ref{fast_outupunmeasured}) (\ref{newsis})-(\ref{saidameasurednew}) are equivalent for $\eta$ sufficiently small, so the demonstration for the original plant (\ref{sistinteg})-(\ref{saidamensura}) follows the same steps as those presented in the proof in \cite[Proposition 1]{AOH:22}, for the case of relative degree one. $\hfill$
\end{proof}

\begin{remark} \label{obs}
Considering that the project is carried out considering a slow time-scale $\eta t$, it is natural that the reference trajectory parameter $p$ is rescaled appropriately, i.e., $\eta p$.
\end{remark}

\end{proposition}

\subsection{Global Convergence}

In this result, it is demonstrated that the multivariable controller based on sliding modes and output feedback drives $z$ towards the region $\mathcal{D}_{\Delta}$, where the unknown maximizer $z^*$ is located, as defined in \textbf{(H6)}. However, this does not guarantee that $z(t)$ stays around $\mathcal{D}_{\Delta}, \,\forall t$. Nevertheless, it is possible to generate oscillations around $y^*$ of the order $\mathcal{O}(\sqrt{\eta}+\varepsilon)$.

\begin{theorem}                                                                           \label{theorem:01}

Consider the system (\ref{sistinteg})--(\ref{saidamensura}), control law (\ref{eqdrakunov})--(\ref{smcerror}), reference trajectory (\ref{modref}) and modulation function (\ref{funcmodgen3}).
Assume that hypotheses $\bf{(H1)-(H6)}$ are satisfied, then: \textbf{(i)} the region $\mathcal{D}_{\Delta}$ in $\bf{(H4) }$ is globally attractive, achieved in finite time and \textbf{(ii)} for $L_{h}$ sufficiently small, oscillations around the maximum value $y^*$ of $y$ can be made in the order $ \mathcal{O}(\sqrt{\eta}+\varepsilon)$. Since the $y_m$ signal can be saturated at (\ref{modref}), all closed-loop signals remain uniformly bounded.
\end{theorem}

\begin{proof}
Similar to the demonstration of Proposition 1, the demonstration of this theorem follows the steps presented in the proof in \cite[Theorem 1]{AOH:22}, for the case of relative degree one, combined with the demonstration in \cite{CH:1991,CH:1992}.  $\hfill$
\end{proof}

\section{Illustrative example}

To illustrate, let us take as an example a plant with an unknown objective function, in series with a linear dynamic system described by

\begin{align}
\dot{v}&=u\,, \label{mod1} \\
    \dot{x}&=\left[
    \begin{array}{cc}
0 & 1 \\ -4 &-2
    \end{array}\right]x+
        \left[
    \begin{array}{cc}
        1 & 0 \\ 0 & 1
    \end{array} \right]v\,,  \label{mod2}  \\
        z&=\left[
    \begin{array}{cc}
        1 &  0 \\ 0 & 1
    \end{array} \right]x  \,, \label{mod3}
\end{align}
and output function
\begin{align}
    y=h(z)=2-(z_1^2+z_2^2-2\epsilon z_1z_2)\label{fobjetivo1}
\end{align}

The static function (\ref{fobjetivo1}) consists of the particular representation of functions of the type
\begin{align} \label{funcaoconvexa}
y=h(z)=y^*+\frac{1}{2}(z-z^*)^TH(z-z^*)
\end{align}
where $$H= \left[ {\begin{array}{cc}
    2 & 2\epsilon \\
    2\epsilon & 2 \\
   \end{array} } \right]<0$$ is the negative defined Hessian matrix.

   Note that the desired parameters of the objective function (\ref{fobjetivo1}) are $z^*=(0,0)$ and $y^*=2$, for $0<\epsilon<1$, a condition for $h(z)$ has a maximum point. Note that the objective function is convex and based on (\ref{funcaoconvexa}).

   The control law (\ref{eqdrakunov})--(\ref{smcerror}) can be applied with the modulation function defined in (\ref{funcmodgen}). The following simulation parameters were chosen: $p=1$, $p_0=0$, $L_h=0.1$ $\lambda=4$ $\varepsilon=0.02$, $\gamma=0.1$ $\eta=0.01 $ and $T_s=5s$. For $\tau=0.1s$ we have $t=10s$ according to (\ref{time-scaling}).

   Figure \ref{xyconvergency} illustrates the result of the proposed multivariable extremum seeking control scheme that converges to the neighborhood of the optimal point $z^*=(0,0)$ and $y^*=2$, starting from the initial condition $z(0)=(-2,4)$. On the other hand, Figure \ref{planodefase} shows the phase portrait for the initial condition $z_0=(-2,4)$ and respective convergence to the equilibrium point, corroborating the illustration in Figure \ref{xyconvergency}.

\begin{figure}[!htb]
\begin{center}
\includegraphics[width=.43\textwidth]{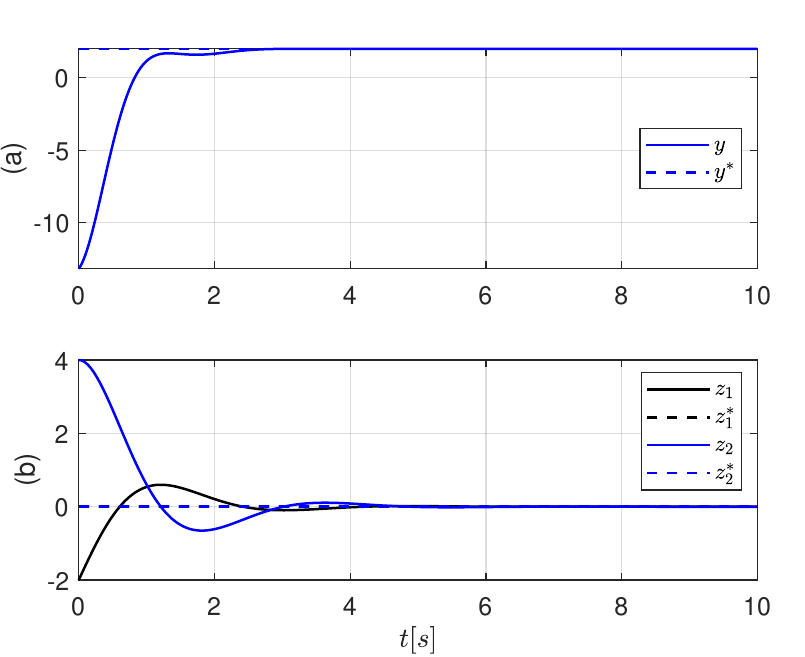}
\caption{Vector parameters $z$ converge to $(0,0)$ starting from the initial condition $z(0)=(-2,4)$ and plant output converges to $y^*=2$.}
\label{xyconvergency}
\end{center}
\end{figure}

\begin{figure}[!htb]
\begin{center}
\includegraphics[width=.43\textwidth]{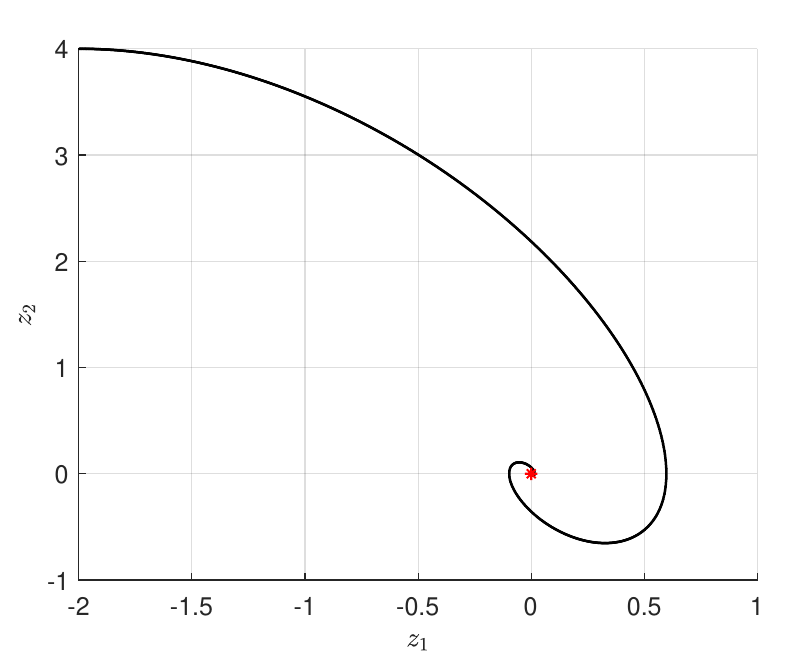}
\caption{Phase portrait shows the convergence to the equilibrium point marked with a red asterisk), starting from the initial condition $z_0=(-2, 4)$.}
\label{planodefase}
\end{center}
\end{figure}

The switching of control signals $u_1$ and $u_2$ is illustrated in Figure \ref{ctrlanddirection}. These signals have gain defined by the modulation function (\ref{funcmodgen}). Furthermore, high-frequency switching is notorious, which could cause the unwanted phenomenon called ``\textit{chattering}''. However, the system (\ref{dinamicalinear})--(\ref{subsist1}) receives the filtered control signals $z$. In the same figure you can observe the behavior of the cyclic search function ($\sigma_1$ and $\sigma_2$), where the last two staggered search periods are highlighted, with the appropriate scaling. It is important to observe alternation in the activation in each half cycle.

\begin{figure}[!htb]
\begin{center}
\includegraphics[width=.43\textwidth]{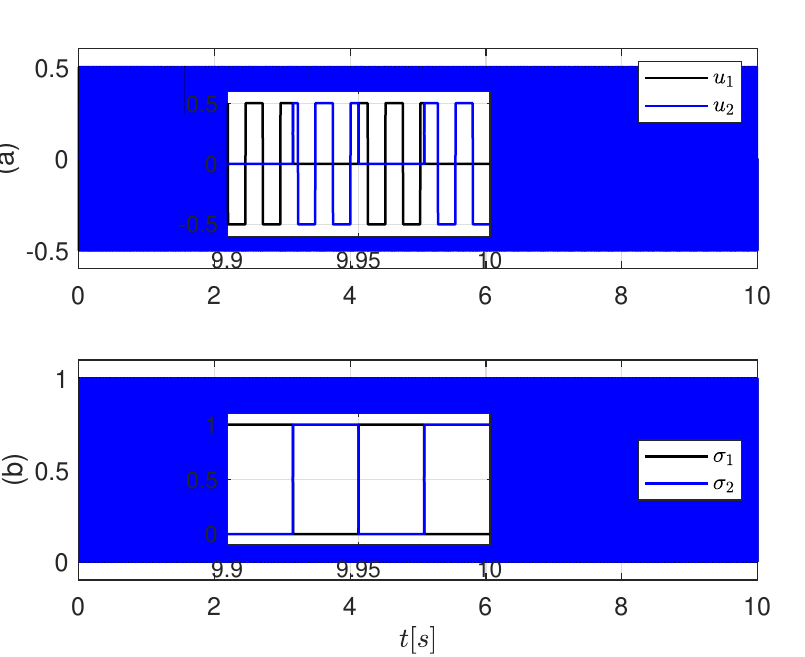}
\caption{(a) control signals $u_1$ and (b) the cyclical search direction $\sigma_1$ and $\sigma_2$, with period $T_s=5 s$ scaled by $\eta=0.01$.}
\label{ctrlanddirection}
\end{center}
\end{figure}

Figure \ref{fig3dimension} illustrates the representation of the objective function in three-dimensional space, indicating the output trajectories for two different initial conditions, $z_0=~(-2, 4)$ in black and $z_0=~(0, 5)$ in blue. Notably, both converge to $y^*=2$.

\begin{figure}[!htb]
\begin{center}
\includegraphics[width=.42\textwidth]{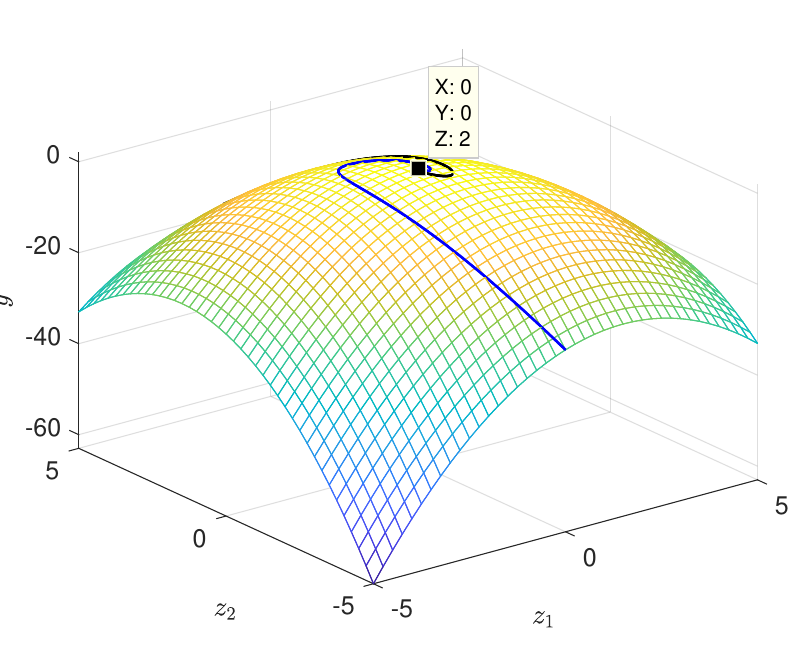}
\caption{Output trajectories going to the optimal point $y^*=2$, from two initial conditions $z_0=~(-2, 4)$ and $z_0=~(0, 5)$.}
\label{fig3dimension}
\end{center}
\end{figure}

\break
\newpage
\section{Conclusions}                                       \label{section5}

In this paper, a multivariable extremum seeking control strategy based on a periodic switching function and cyclic search for linear dynamic maps with arbitrary relative degree was presented. Relative degree is mitigated through temporal scaling. A convex and non-linear objective function was considered. The resulting approach guarantees global convergence of the system output to a small neighborhood of the extremum using only output feedback. Through simulation, it was possible to demonstrate the performance of the proposed real-time extremum seeking control system.

%


\end{document}